\documentclass[11pt]{article}
\usepackage{graphicx}
\usepackage{amsthm, amsmath, amssymb, tikz, bm}
\usepackage{mathtools}
\usepackage{xifthen}
\usepackage{comment}

\usepackage[margin=1.1in]{geometry} 
\usepackage{thm-restate}
\usepackage[shortlabels]{enumitem}
\usepackage{todonotes}
\usepackage[colorlinks=true, linkcolor=blue,citecolor=blue,
urlcolor=blue]{hyperref}

\usetikzlibrary{calc,shapes, backgrounds}
\usetikzlibrary[patterns]
\usetikzlibrary{arrows.meta}
\usetikzlibrary{decorations.markings}

\newtheorem{theorem}{Theorem}
\newtheorem{lemma}[theorem]{Lemma}

\newtheorem{observation}{Observation}

\newcommand{\E}{{\rm E}}

\DeclarePairedDelimiter{\floor}{\lfloor}{\rfloor}

\makeatletter
               {\list{}{\leftmargin=0pt 
                        \labelwidth\z@ \itemindent-\leftmargin
                        }}%
               {\endlist}
\makeatother

\newcommand{\ex}{\text{ex}}
\def\exP{\textrm{ex}_{\mathcal{P}}}

\title{Rainbow planar Tur{\'a}n numbers of cycles}

\author{
Xiaonan Liu \thanks{Department of Mathematics, Vanderbilt University, Nashville, TN, 37240
({\tt xiaonan.liu@vanderbilt.edu}).}}

\begin{document}

\maketitle
\begin{abstract}
The rainbow Tur{\'a}n number of a fixed graph $H$, denoted by ${\ex}^*(n,H)$, is the maximum number of edges in an $n$-vertex graph such that it admits a proper edge coloring with no rainbow $H$. We study this problem in planar setting. The rainbow planar Tur{\'a}n number of a  graph $H$, denoted by ${\exP}^*(n,H)$, is the maximum number of edges in an $n$-vertex planar graph such that it has a proper edge coloring with no rainbow $H$. We consider the rainbow planar Tur{\'a}n number of cycles. Since $C_3$ is complete, $\exP^*(n, C_3)$ is exactly its planar Tur{\'a}n number, which is $2n-4$ for $n\ge 3$. We show that $\exP^*(n, C_4)=3n-6$ for $n=k^2-3k+2$ where $k\ge 5$, and $\exP^*(n,C_k)=3n-6$ for all $k\ge 5$ and $n\ge 3$.
\end{abstract}

\section{Introduction}
All graphs in this paper are finite and simple. Given a graph $H$, the classical \emph{Tur\'an number} of $H$, denoted by $\ex(n,H)$, is the maximum number of edges in an $n$-vertex graph that does not contain $H$ as subgraph. The study on the Tur\'an number of graphs is a central topic in extremal graph theory. Motivated by extremal problems in additive number theory and graph coloring, Keevash, Mubayi, Sudakov and Verstra\"ete~\cite{KMSV2007} initiated the study on the \emph{rainbow Tur\'an problems}.
Given a graph $G$, a \emph{proper edge-coloring} of $G$ is an edge coloring of $G$ such that no two adjacent edges of $G$ receive the same color, and a \emph{$k$-edge-coloring} of $G$ for some integer $k$ is a an edge coloring of $G$ using $k$ colors.  An edge-colored graph $H$ is called \emph{rainbow} if no two edges of $H$ receive the same color. For a fixed graph $H$, the \emph{rainbow Tur\'an number} of $H$, denoted by $\ex^*(n,H)$, is defined as the maximum number of edges in an $n$-vertex graph that has a proper edge-coloring with no rainbow copy of $H$. It is obvious that $\ex^*(n,H)\geq \ex(n,H)$. In~\cite{KMSV2007}, Keevash, Mubayi, Sudakov and Verstra\"ete determined $\ex^*(n,H)$ asymptotically for non-bipartite graphs $H$. In particular, they showed that for any fixed $H$ and sufficiently large $n$, 
$$\ex(n,H)\leq \ex^*(n,H)\leq \ex(n,H)+o(n^2).$$
The \emph{chromatic number} of a graph $G$, denoted by $\chi(G)$, is the smallest number of colors that are needed to color all the vertices in $G$ such that no two adjacent vertices of $G$ receive the same color. A graph $H$ is called \emph{color-critical} if it contains an edge $e$ such that $\chi(H\backslash e) = \chi(H)-1$. For example, complete graphs and odd cycles are all color-critical.  Keevash, Mubayi, Sudakov and Verstra\"ete in~\cite{KMSV2007} proved  that
$\ex(n,H)= \ex^*(n,H)$ if $H$ is color-critical. For even cycles, they~\cite{KMSV2007} showed that for all $k\geq 2$, $\ex^*(n,C_{2k})\geq cn^{1+1/k}$ for some $c>0$ and proved that $\ex^*(n,C_{2k})= O(n^{1+1/k})$ for $k\in\{2,3\}$. They conjectured that the same asymptotic upper bound on $\ex^*(n,C_{2k})$ holds for all $k\geq 2$, which was recently confirmed by Janzer~\cite{Janzer2023}.
The rainbow Tur\'an number of paths were recently studied in ~\cite{JPS2017, JR2019, EGM2019, HP2021, Halfpap2023}. Further results on (generalized) rainbow Tur\'an problems can be found in~\cite{Janzer2023, GMMP2022} and reference therein. 

In a more sparse setting, recently, there has been extensive research on Tur\'an problems in planar graphs. For a fixed graph $H$, the \emph{planar Tur\'an number} of $H$, denoted by $\exP(n,H)$, is defined as the maximum number of edges in an $n$-vertex planar graph without containing $H$ as a subgraph. It is clear from Euler's formula that $\exP(n,C_3)=2n-4$ for $n\ge 3$. Dowden~\cite{Dowden2016} proved that $\exP(n,C_4)\leq \frac{15(n-2)}{7}$ for all $n\geq 4$ and $\exP(n,C_5)\leq \frac{12n-33}{5}$ for all $n\geq 11$. Ghosh, Gy\H{o}ri, Martin, Paulos and Xiao~\cite{GGMPRX2022} showed that $\exP(n,C_6)\leq \frac{5n-14}{2}$ for all $n\geq 18$. All bounds above are tight for infinitely many $n$. In the same paper, Ghosh et al.~\cite{GGMPRX2022} conjectured that $\exP(n, C_{k})\leq \frac{3(k-1)}{k}n-\frac{6(k+1)}{k}$ for all $k\geq 7$ and sufficiently large $n$, which was disproved by Cranston, Lidick\'y, Liu and Shantanam~\cite{CLLS2022} for all $k\geq 11$ (see also~\cite{Lan-Song2022}). However, the conjecture of Gosh et al.~\cite{GGMPRX2022} for $7\leq k\leq 10$ may still hold and was recently verified to be true for $k=7$ very recently by Shi, Walsh and Yu~\cite{SWY2025} and independently by Gy\H{o}ri, Li and Zhou~\cite{GLZ2023}. Confirming a conjecture of Cranston et al. \cite{CLLS2022}, Shi, Walsh and Yu~\cite{SWY2025_dense} proved an upper bound of $3n-6-Dn/\ell^{\log_23}$ for $\exP(n, C_{\ell})$ for all $\ell, n\ge 4$, where $D$ is some constant. For recent results on the planar Tur\'an number of graphs other than cycles, see the recent survey by Lan, Shi and Song~\cite{LSS2021} and references therein.

Very recently, motivated by the recent active developments on the rainbow Tur\'an number and planar Tur\'an number, Gy\H{o}ri, Martin, Paulos, Tompkins and Varga~\cite{GMPTV2025} initiated the study on the \emph{rainbow planar Tur\'an problems}. Given a fixed graph $H$, the \emph{rainbow planar Tur\'an number} of $H$, denoted by $\exP^*(n,H)$, is defined as the maximum number of edges in an $n$-vertex planar graph that has a proper edge-coloring with no rainbow copy of $H$. A \emph{planar triangulation} is an edge-maximal planar graph such that every face of its plane embedding is bounded by a triangle. By Euler's formula, an $n$-vertex planar triangulation ($n\ge 3$) has exactly $3n-6$ edges.
Note that for each $n\geq 3$, there exist planar triangulations on $n$ vertices that can be properly edge-colored with at most $6$ colors. For example, Figure~\ref{fig:6colors} is a properly $6$-edge-colored planar triangulation on even number vertices. Moreover, deleting the lower left vertex or the upper right vertex results in a properly $6$-edge-colored planar triangulation on odd number vertices. It follows that if a graph $H$ has more than six edges then $\exP^*(n,H)=3n-6$ for $n\ge 3$. Gy\H{o}ri et al.~\cite{GMPTV2025} made a systematic study on the rainbow planar Tur\'an number pf paths. In particular, they observed that $\exP^*(n,P_3) = \floor{n/2}$ and showed that $\exP^*(n,P_4) = \exP^*(n,P_5) = \floor{3n/2}$ for all $n\geq 4$.  We know that $\exP^*(n,P_k)=3n-6$ for all $k\ge 8$ and $n\ge 3$. The cases for $P_6, P_7$   remain open. Gy\H{o}ri et al.~\cite{GMPTV2025} conjectured that
$2n-O(1)\leq \exP^*(n,P_6)\leq 2n$ and $5n/2-O(1)\leq \exP^*(n,P_7)\leq 5n/2$. He and Liu in~\cite{HL2024} considered the rainbow planar Tur{\'a}n number for some double stars, $S_{1,k}$ for all $k$ except $k=5$ and $S_{2,2}$, where $S_{s,k}$ denotes the graph obtained by taking an edge with $s$ vertices joining one of its end vertices and $k$ vertices joining the other end vertex. As $S_{1,5}$ has seven edges, we know that $\exP(n,S_{1,5})=3n-6$ for $n\geq 3$.

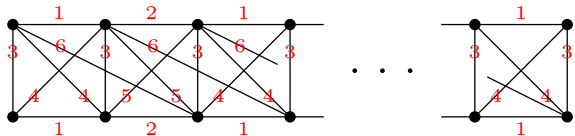
\begin{figure}[htb]
    \centering
    \resizebox{8cm}{!}
    {\begin{tikzpicture}[scale=1, wvertex/.style={circle, draw=red, fill=red, scale=0.3}, bvertex/.style={circle, draw=black, fill=black, scale=0.3},rvertex/.style={circle, draw=red, fill=red, scale=0.3}, sbvertex/.style={circle, draw=black, fill=black, scale=0.1}]
\node [bvertex](u) at (-3.5,0.5) {};
  \node [bvertex] (u1) at (-2.5,0.5) {};
   \node [bvertex] (u2) at (-1.5,0.5) {};
    \node [bvertex] (u3) at (-0.5,0.5) {};
     \node [bvertex] (u-1) at (1.5,0.5) {};
      \node [bvertex] (u0) at (2.5,0.5) {};
 \node [bvertex] (v) at (-3.5,-0.5) {};   
  \node [bvertex] (v1) at (-2.5,-0.5) {};
   \node [bvertex] (v2) at (-1.5,-0.5) {};
    \node [bvertex] (v3) at (-0.5,-0.5) {};
     \node [bvertex] (v-1) at (1.5,-0.5) {};
      \node [bvertex] (v0) at (2.5,-0.5) {}; 

  \draw (u)--node[font=\tiny, above=-0.8mm,text=red]{$1$}(u1);
    \draw (v)--node[font=\tiny, below=-0.8mm,text=red]{$1$}(v1);
      \draw (u1)--node[font=\tiny, above=-0.8mm,text=red]{$2$}(u2);
     \draw  (v-1)--node[font=\tiny, below=-0.8mm,text=red]{$1$}(v0);
      \draw(u-1)--node[font=\tiny, above=-0.8mm,text=red]{$1$}(u0);
    \draw (v1)--node[font=\tiny, below=-0.8mm,text=red]{$2$}(v2);
   
   \draw(v2)--node[font=\tiny, below=-0.8mm,text=red]{$1$}(v3);
      \draw  (u2)--node[font=\tiny, above=-0.8mm,text=red]{$1$}(u3);
      
      \draw (v2)--node[font=\tiny,  below left=1mm,text=red]{$4$}(u3);
        \draw (v3)--node[font=\tiny, below right=1mm,text=red]{$4$}(u2);

           \draw (v)--node[font=\tiny,  below left=1mm,text=red]{$4$}(u1);
        \draw (v1)--node[font=\tiny, below right=1mm,text=red]{$4$}(u);
        
         \draw (u)--node[font=\tiny, above,midway, text=red]{$3$}(v); 
      \draw (u3)--node[font=\tiny, above,midway,text=red]{$3$}(v3);
      \draw (u-1)--node[font=\tiny, above,midway,text=red]{$3$}(v-1);
      \draw (u2)--node[font=\tiny,  above,text=red]{$3$}(v2);
      \draw (u0)--node[font=\tiny, above, midway,text=red]{$3$}(v0);
      \draw (u1)--node[font=\tiny, above,midway,text=red]{$3$}(v1);
      
      \draw (u1)--node[font=\tiny, below right=1mm,text=red]{$5$}(v2);
    \draw (v-1)--node[font=\tiny, below right=1mm,text=red]{$4$}(u0);
    \draw (v1)--node[font=\tiny, below left=1mm,text=red]{$5$}(u2);
      \draw (u-1)--node[font=\tiny, below left=1mm,text=red]{$4$}(v0);
      
      \node (I1) at (-0.5,0) {};
      \draw [black] (u2)--node[font=\tiny, midway,text=red]{$6$}(I1);
      \node (I2) at (1.5,0) {};
      \draw [black] (v0)--(I2);
  \draw (u1)--node[font=\tiny, above left, xshift=-8pt, yshift=2pt,text=red]{$6$}(v3); 
  \draw (u)--node[font=\tiny, above left, xshift=-8pt, yshift=2pt,text=red]{$6$}(v2); 
      \node(I3) at (0,0.5){};
      \draw (u3)--(I3);

        \node(I4) at (0,-0.5){};
      \draw (v3)--(I4);

        \node(I5) at (1,0.5){};
      \draw (u-1)--(I5);

      \node(I6) at (1,-0.5){};
      \draw(v-1)--(I6);

\node[sbvertex] at (0.5,0) {};
\node[sbvertex] at (0.2,0) {};
\node[sbvertex] at (0.8,0) {};
\end{tikzpicture}}
     \caption{A properly $6$-edge-colored planar triangulation.}
    \label{fig:6colors}
\end{figure}

In this paper, we investigate the rainbow planar Tur\'an number of cycles, which is the next natural  planar graph class to consider. For $C_3$, since a $3$-cycle is a complete graph, we have that $\exP^*(n, C_3)=\exP(n,C_3)=2n-4$ for $n\geq 3$. The remaining open cases are $C_4, C_5, C_6$. It is surprising that there exists an $n$-vertex planar triangulation such that it has a proper edge-coloring with no rainbow $C_4$ for infinitely many $n$. For $C_5$ and $C_6$, we also find an $n$-vertex planar triangulation (see the even case in Figure~\ref{fig:6colors}) admitting a proper edge-coloring with no rainbow $C_5$ or $C_6$ for each $n\geq 3$.

\begin{theorem}\label{thm:C4}
For each integer $k\ge 5$ and $n=k^2-3k+2$, $\exP^*(n, C_4)=3n-6$.
\end{theorem}
While we cannot give constructions of planar triangulations having a proper edge-coloring with no rainbow $C_4$ for other values $n$, we have the following result for those graphs. 
\begin{theorem}\label{thm:C4_other}
Let $G$ be a planar triangulation on at least five vertices such that $G$ has a proper edge-coloring  containing no rainbow $C_4$. Then $G$ has minimum degree $5$ and $G$ is $4$-connected.
\end{theorem}
\begin{theorem}\label{thm:C5}
For each $n\ge 3$ and each $k\ge 5$, $\exP^*(n, C_k)=3n-6$.
\end{theorem}

This paper is organized as follows. In section 2, we first show Theorem~\ref{thm:C4_other}, and then we describe an $n$-vertex planar triangulation with a proper edge-coloring containing no rainbow $C_4$ for $n=k^2-3k+2$ with each $k\ge 5$.  In section 3, we give an $n$-vertex triangulation with a proper edge-coloring  such that it contains no rainbow $C_5$ or $C_6$, which implies Theorem~\ref{thm:C5}.

\medskip
We conclude this section with some terminology and notation. For any positive integer
$k$ let $[k] = \{1, 2, \ldots , k\}$, and for positive integers $s,t$ with $s<t$ let $[s,t]=\{s, s+1, \ldots, t\}$. 

Let $G$ be a graph. For $v\in V(G)$, we use $N_G(v)$ (respectively, $N_G[v]$) to denote the neighborhood (respectively, closed neighborhood) of $v$, and use $d_G(v)$ to denote $|N_G(v)|$. For distinct vertices $u,v$ of $G$, we use $d_G(u,v)$ to denote the distance between $u$ and $v$. (If there is no confusion we omit the reference to $G$.) For any $S\subseteq V(G)$, we use $G[S]$ to denote the subgraph of $G$ induced by $S$, and let $G-S = G[V(G)\backslash S]$. For a subgraph $T$ of $G$, we often write $G-T$ for $G-V(T)$ and write $G[T]$ for $G[V(T)]$. A path (respectively, cycle) is often represented as a sequence (respectively, cyclic sequence) of vertices, with consecutive vertices 
being adjacent. A cycle $C$ in a graph $G$ is said to be \textit{separating} if the graph obtained from $G$ by deleting vertices in $C$ is not connected. 

Let $G$ be a plane graph. For a cycle $C$ in $G$, we use $\text{Int}_G(C)$ and $\text{Ext}_G(C)$ to denote the interior and the exterior of this cycle $C$ in $G$, respectively. 
\section{Rainbow planar {Tur{\'a}n} number of $C_4$}
Before we prove Theorem~\ref{thm:C4_other}, we have the following observation for properly edge-colored planar triangulations with no rainbow $C_4$.
\begin{observation}\label{obs:rainbow_neighbor}
  Let $G$ be a properly edge-colored planar triangulation on at least four vertices with no rainbow $C_4$. Then for any vertex $v$ in $G$, $G[N(v)]$ has a rainbow cycle using all vertices of $N(v)$.   
\end{observation}
\begin{proof}
    Suppose $G$ is a planar triangulation on at least four vertices and it has a proper edge-coloring $c$ with no rainbow $C_4$. Let $v$ be a vertex in $G$. Since $G$ is edge-maximal and $G\ne K_3$, $v$ has degree at least three and all neighbors of $v$ are contained in a cycle. Assume that $d(v)=k$ for some $k\ge 3$ and $N(v)=\{v_1, v_2, \ldots, v_k\}$, where $C=v_1v_2\ldots v_k v_1$ is a cycle in $G$. To show $G[N(v)]$ has a rainbow cycle containing all vertices of $N(v)$, it suffices to show that $C$ is rainbow. We may now assume $d(v)=k\geq 4$ as a properly edge-colored triangle is always rainbow. Since $G$ is properly edge-colored, we may assume that the edge $vv_i$ is colored with the color $i$ for each $i\in [k]$. For each $i\in [k]$, let $D_i:=vv_{i}v_{i+1}v_{i+2}v$, where the indices are same under modulo $k$. Note that each $D_i$ is a $4$-cycle and $G$ has no rainbow $C_4$. Since $D_1=vv_1v_2v_3v$ is not rainbow, either $c(v_1v_2)=c(vv_3)=3$ or $c(v_2v_3)=c(vv_1)=1$. Without loss of generality, we may assume the edge $v_2v_3$ is colored with the color $1$. Then observe that $D_2=vv_2v_3v_4v$ is a $4$-cycle with $c(vv_2)=2, c(v_2v_3)=1, c(vv_4)=4$. It follows that the edge $v_3v_4$ has the same color as $vv_2$, i.e., $c(v_3v_4)=c(vv_2)=2$. Similarly, we have that $c(v_iv_{i+1})=i-1$ for each $i\in [2,k-1]$ and $c(v_kv_1)=k-1, c(v_1v_2)=k$. Thus $C$ is rainbow, and this completes the proof.
\end{proof}
\begin{proof}[Proof of Theorem~\ref{thm:C4_other}] Suppose $G$ is a planar triangulation such that $n=|V(G)|\ge 5$ and it has a proper edge-coloring $c$ with no rainbow $C_4$. It follows from Observation~\ref{obs:rainbow_neighbor} that $G$ has no degree $4$ vertices. Since $G$ has at least five vertices and exactly $3n-6$ edges, the maximum degree of $G$ is at least four. Hence, $G$ has a vertex of degree at least $5$.
To show $G$ has minimum degree $5$, we need to show $G$ has no vertex of degree $3$. Suppose $G$ has a degree $3$ vertex, say $u$. Observe that all neighbors of $u$ have degree at least $5$ in $G$. Let $v$ be a neighbor of $u$. Suppose $d(v)=k$ for some $k\geq 5$ and $N(v)=\{v_1, v_2, \ldots, v_k\}$ such that $C=v_1v_2\ldots v_k$ is a cycle in $G$. Without loss of generality, we may assume $v_1=u$ and so $v_kv_2\in E(G)$. By Observation~\ref{obs:rainbow_neighbor}, we may assume that $c(vv_i)=c(v_{i+1}v_{i+2})=i$ for each $i\in [k]$, where the indices are same under modulo $k$. We consider the color for the edge $v_2v_k$. Note that $c(vv_2)=2,c(v_2v_3)=1, c(v_kv_1)=k-1$. Hence $c(v_2v_k)\notin \{1, 2,k-1\}$. This implies the $4$-cycle $v_kv_1vv_2v_k$ is rainbow, giving a contradiction. Therefore, every vertex in $G$ has degree at least $5$. Since every planar graph has minimum degree at most $5$, it follows that $G$ has minimum degree $5$.
\begin{figure}[htb]
    \centering
    \resizebox{3cm}{!}
    {\begin{tikzpicture}[scale=1, wvertex/.style={circle, draw=red, fill=red, scale=0.3}, bvertex/.style={circle, draw=black, fill=black, scale=0.3},rvertex/.style={circle, draw=red, fill=red, scale=0.3}, sbvertex/.style={circle, draw=black, fill=black, scale=0.1}]
    
    \node [bvertex, label={[font=\small]  left, above:$v_1$}] (v1) at (90:2) {};

    \node [bvertex, label={[font=\small]  below:$v_{2}$}](v2) at (330:2) {};
    \node [bvertex, label={[font=\small] below:$v_{3}$}](v3) at (210:2) {};
    \node [bvertex, label={[font=\small]  below:$u_{3}$}](u3) at (30:0.5) {};
    \node [bvertex, label={[font=\small] below:$u_2$}](u2) at (150:0.5) {};
    \node [bvertex, label={[font=\small] above:$u_1$}](u1) at (270:0.5) {};

 \node [bvertex, label={[font=\small] right, above:$w$}](w) at (30:1.5) {};

   \draw (v1) -- node[font=\tiny, midway,text=red]{$1$}(v2) --node[font=\tiny, midway, text=red]{$2$} (v3)--node[font=\tiny, midway, text=red]{$3$} (v1);

\draw (v1)--node[font=\tiny, midway,text=red]{$2$}(u3)--(v2);
\draw (v2)--node[font=\tiny, midway,text=red]{$3$}(u1)--(v3);
\draw (v3)--node[font=\tiny, midway,text=red]{$1$}(u2)--(v1);
\draw (v1)--(w)--(v2);

        
\end{tikzpicture}}
     \caption{The separating triangle $T$}
    \label{fig:triangle}
\end{figure}
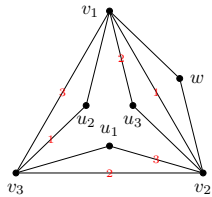

Next we show that $G$ is $4$-connected. Suppose not. Then $G$ contains a separating triangle. Consider a plane embedding of $G$. We choose a separating triangle $T:=v_1v_2v_3v_1$ of $G$ such that its interior is vertex-minimum. We claim that $v_1, v_2,v_3$ cannot have a common neighbor in the interior of $T$. Suppose $u$ is a common neighbor of $v_1, v_2, v_3$ in $\text{Int}(T)$. Then we have that either $u$ is a degree $3$ vertex in $G$ or there is a separating triangle with fewer vertices in its interior, a contradiction. Let $u_{6-i-j}$ denote the common neighbor of $v_i$ and $v_j$ in $\text{Int}(T)$ for $i,j$ with $1\le i <j \le 3$. Note that $v_1, v_2$ must have a common neighbor, say $w$, in $\text{Ext}(T)$. We may assume that $c(v_1v_2)=1, c(v_2v_3)=2$ and $c(v_3v_1)=3$. Since the $4$-cycle $v_1u_3v_2v_3v_1$ is not rainbow, either $c(v_1u_3)=c(v_2v_3)=2$ or $c(v_2u_3)=c(v_3v_1)=3$. without loss of generality, we may assume that $c(v_1u_3)=c(v_2v_3)=2$. Note that $v_1v_2v_3u_2v_1$ is not rainbow and $c(v_1u_2)\neq c(v_1u_3)=2=c(v_2v_3)$. It follows that $c(v_3u_2)=c(v_1v_2)=1$. Similarly, we have that $c(v_2u_1)=3$. Thus $c(v_1w)\notin\{1,2,3\}$ and $c(v_2w)\notin\{1,2,3\}$. This implies that the $4$-cycle, $v_1wv_2v_3v_1$, is rainbow, a contradiction. Therefore, $G$ has no separating triangle and so $G$ is $4$-connected.
\end{proof}
Now we want to give a planar triangulation such that it has a proper edge-coloring with no rainbow $C_4$. By Theorem~\ref{thm:C4_other}, we know that this planar triangulation $G$ should have minimum degree $5$ and connectivity at least $4$. If $G$ is $5$-connected, will this be helpful? We know that $5$-connected planar triangulations have no separating $4$-cycles, and so every $4$-cycle in $5$-connected planar triangulation is formed by two adjacent facial triangles. Hence $G$ has exactly $3|V(G)|-6$ many $4$-cycles for $5$-connected $G$. This together with Observation~\ref{obs:rainbow_neighbor} motivates us to give the following construction. 

\begin{figure}[htb]
    \centering
    \resizebox{6cm}{!}
    {\begin{tikzpicture}[scale=1, wvertex/.style={circle, draw=red, fill=red, scale=0.3}, bvertex/.style={circle, draw=black, fill=black, scale=0.3},rvertex/.style={circle, draw=red, fill=red, scale=0.3}, sbvertex/.style={circle, draw=black, fill=black, scale=0.1}]
\node[sbvertex] at (0.5,2.4) {};
\node[sbvertex] at (0.3,2.4) {};
\node[sbvertex] at (0.1,2.4) {};

\node[sbvertex] at (1,1.5) {};
\node[sbvertex] at (0.75,1.5) {};
\node[sbvertex] at (0.5,1.5) {};

\node[sbvertex] at (1,0.5) {};
\node[sbvertex] at (0.75,0.5) {};
\node[sbvertex] at (0.5,0.5) {};


\node[sbvertex] at (1,-1.5) {};
\node[sbvertex] at (0.75,-1.5) {};
\node[sbvertex] at (0.5,-1.5) {};

\node[sbvertex] at (0.6,-2.4) {};
\node[sbvertex] at (0.4,-2.4) {};
\node[sbvertex] at (0.2,-2.4) {};

\node[sbvertex] at (-2,-0.3) {};
\node[sbvertex] at (-2, -0.5) {};
\node[sbvertex] at (-2,-0.7) {};

\node[sbvertex] at (-1,-0.3) {};
\node[sbvertex] at (-1, -0.5) {};
\node[sbvertex] at (-1,-0.7) {};

\node[sbvertex] at (2.5,-0.3) {};
\node[sbvertex] at (2.5, -0.5) {};
\node[sbvertex] at (2.5,-0.7) {};

    \node [bvertex, label={[font=\small] above:$v_0$}] (v0) at (0,3) {};
    \node [bvertex, label={[font=\small] below:$v_{k-2}$}](v-2) at (0,-3) {};
    
    \node [bvertex, label={[font=\small] left:$v_{1,1}$}](v11) at (-2.5,2) {};
    \node [bvertex](v12) at (-1.5,2) {};
    \node [bvertex](v13) at (-0.5,2) {};
    \node [bvertex](v1-1) at (2,2) {};
    \node [bvertex, label={[font=\small] right:$v_{1,k}$}](v10) at (3,2) {};

  \node [bvertex, label={[font=\small] left:$v_{2,1}$}](v21) at (-2.5,1) {};
    \node [bvertex](v22) at (-1.5,1) {};
    \node [bvertex](v23) at (-0.5,1) {};
    \node [bvertex](v2-1) at (2,1) {};
    \node [bvertex, , label={[font=\small] right:$v_{2,k}$}](v20) at (3,1) {};

      \node [bvertex, label={[font=\small] left:$v_{3,1}$}](v31) at (-2.5,0) {};
    \node [bvertex](v32) at (-1.5,0) {};
    \node [bvertex](v33) at (-0.5,0) {};
    \node [bvertex](v3-1) at (2,0) {};
    \node [bvertex,, label={[font=\small] right:$v_{3,k}$}](v30) at (3,0) {};

       \node [bvertex,label={[font=\small] left: $v_{k-4,1}$}](v-41) at (-2.5,-1) {};
    \node [bvertex](v-42) at (-1.5,-1) {};
    \node [bvertex](v-43) at (-0.5,-1) {};
    \node [bvertex](v-4-1) at (2,-1) {};
    \node [bvertex, label={[font=\small] right:$v_{k-4,k}$}](v-40) at (3,-1) {};   

   \node [bvertex, label={[font=\small] left:$v_{k-3,1}$}](v-31) at (-2.5,-2) {};
    \node [bvertex](v-32) at (-1.5,-2) {};
    \node [bvertex](v-33) at (-0.5,-2) {};
    \node [bvertex](v-3-1) at (2,-2) {};
    \node [bvertex, label={[font=\small] right:$v_{k-3,k}$}](v-30) at (3,-2) {};

    \draw[-] (v11) .. controls (-0.75,4) and (1.25,4) .. (v10);
    
    \draw[-] (v-31) .. controls (-0.75,-4) and (1.25,-4) .. (v-30);
    
   \draw (v11) -- (v12) -- (v13);
    \draw (v21) -- (v22) -- (v23);
     \draw (v31) -- (v32) -- (v33);
      \draw (v-41) -- (v-42) -- (v-43);
       \draw (v-31) -- (v-32) -- (v-33);

    \draw (v1-1)--(v10);
       \draw (v2-1)--(v20);
          \draw (v3-1)--(v30);
             \draw (v-4-1)--(v-40);
                \draw (v-3-1)--(v-30);

\draw (v0)--(v11)--(v21)--(v31);
\draw (v0)--(v12)--(v22)--(v32);
\draw (v0)--(v13)--(v23)--(v33);
\draw (v0)--(v1-1)--(v2-1)--(v3-1);
\draw (v0)--(v10)--(v20)--(v30);

\draw (v-2)--(v-31)--(v-41);
\draw (v-2)--(v-32)--(v-42);
\draw (v-2)--(v-33)--(v-43);
\draw (v-2)--(v-3-1)--(v-4-1);
\draw (v-2)--(v-30)--(v-40);

\draw (v21)--(v12);
\draw (v31)--(v22)--(v13);
\draw (v32)-- (v23);
\draw (v2-1)--(v10);
\draw (v3-1)-- (v20);
\draw (v-31)--(v-42);
\draw (v-32)--(v-43);
\draw (v-3-1)--(v-40);
        
\end{tikzpicture}}
     \caption{Part of $H_k$.}
    \label{fig:C_4}
\end{figure}
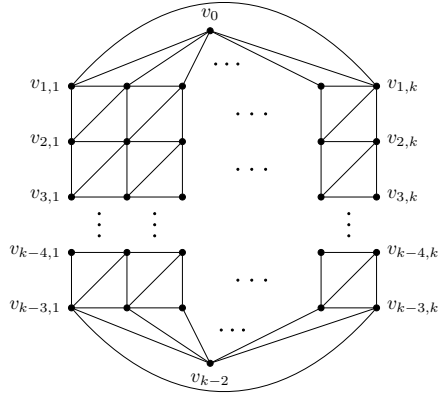

We define a planar triangulation $H_k$ for each $k\ge 5$ such that $H_k$ has $k^2-3k+2$ vertices. Let $$V(H_k):=\{v_0, v_{k-2}\} \cup \{v_{i,j}: i\in [k-3], j\in [k]\}.$$
Now we define the edges of $H_k$. We may assume that all numbers for the index $j$ are same under modulo $k$. Let $v_{i,j}$ be adjacent to $v_{i, j+1}$ for every $i\in [k-3]$, i.e., the vertices $v_{i,1}, v_{i,2},\ldots v_{i,k}$ form  a $k$-cycle in $H_k$, denoted by $D_i$. Moreover, for each $i\in [k-4]$ we let $v_{i,j} v_{i+1,j}\in E(H_k)$ and $v_{i,j}v_{i+1,j-1}\in E(H_k)$. Let $v_0$ be adjacent to all vertices of $D_1$ and $v_{k-2}$ be adjacent to all vertices of $D_{k-3}$. For convenience, let $v_{0,j}=v_0$ and $v_{k-2,j}=v_{k-2}$ for each $j\in [k]$. Hence, $v_{0,j}v_{1,j}=v_0v_{1,j}$ and $v_{k-3,j}v_{k-2,j}=v_{k-3,j}v_{k-2}$.

\begin{observation}
For each $k\ge 5$, $H_k$ has the following properties.
    \begin{itemize}
      \item [(i)] $H_k$ is a $5$-connected planar triangulation on $(k-3)k+2$ vertices.
    \item [(ii)] $v_0,v_{k-2}$ both have degree $k$, $v_{i,j}$ has degree $5$ for $i\in \{1, k-3\}$, and $v_{i,j}$ has degree $6$ for $i\in [2, k-4]$ if $k\ge 6$.
    \end{itemize}
\end{observation}
Theorem~\ref{thm:C4} follows from the following result.
\begin{lemma}
    $H_k$ has a proper edge-coloring with no rainbow $C_4$.
\end{lemma}
\begin{proof}
We give a $(k+2)$-edge-coloring of $H_k$, $\sigma: E(H_k) \to [k]\cup \{a,b\}$, which is defined as follows. 
$$
\sigma(e) = \begin{cases} 
                        a & \textrm{ if $e=v_{i,j}v_{i+1,j-1}$ and $i\in [k-4]$ is odd,}\\
                        b &\textrm{ if $e=v_{i,j}v_{i+1,j-1}$ and $i\in [k-4]$ is even,}\\
                        t \textrm{ (where $t\in [k]$)} & \textrm{ if $e=v_{i,t}v_{i, t+1}$ for $i\in [k-3]$ or $e=v_{i-1,t-i}v_{i, t-i}$ for $i\in [k-2]$.}
                    \end{cases}
$$
We claim that $\sigma$ is a proper edge coloring of $H_k$ and it has no rainbow $C_4$. To prove that $\sigma$ is proper, it suffices to show that for each vertex $v$ in $H_k$, all the edges incident with $v$ receive distinct colors. We may assume the second indices for the vertices in $H_k$, as well as the colors labeled as integers, are same under modulo $k$. 
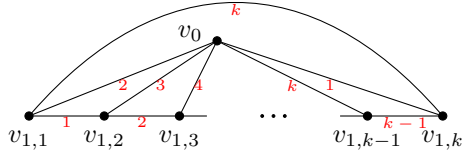
\begin{figure}[htb]
    \centering
    \begin{tikzpicture}[scale=1, wvertex/.style={circle, draw=red, fill=red, scale=0.3}, bvertex/.style={circle, draw=black, fill=black, scale=0.3},rvertex/.style={circle, draw=red, fill=red, scale=0.3}, sbvertex/.style={circle, draw=black, fill=black, scale=0.1}]
    
    \node [bvertex, label={[font=\small] left, above=1mm:$v_0$}] (v0) at (0,3) {};

    \node [bvertex, label={[font=\small] below:$v_{1,1}$}](v11) at (-2.5,2) {};
    \node [bvertex,, label={[font=\small] below:$v_{1,2}$}](v12) at (-1.5,2) {};
    \node [bvertex, label={[font=\small] below:$v_{1,3}$}](v13) at (-0.5,2) {};
    \node [bvertex, label={[font=\small] below:$v_{1,k-1}$}](v1-1) at (2,2) {};
    \node [bvertex, label={[font=\small] below:$v_{1,k}$}](v10) at (3,2) {};
\node (I1) at (0,2){};
\node (I2) at (1.5,2) {};
\draw (v13)--(I1);
\draw (v1-1)--(I2);
 \node[sbvertex] at (0.75,2){};
 \node [sbvertex] at (0.9,2){};
 \node [sbvertex] at (0.6,2){};

    \draw[-] (v11) .. controls (-0.75,4) and (1.25,4) .. node[font=\tiny, below=-1.2mm,text=red]{$k$}(v10);

   \draw (v11) -- node[font=\tiny, below=-1.2mm,text=red]{$1$}(v12) --node[font=\tiny, below=-1.2mm, text=red]{$2$} (v13);

    \draw (v1-1)-- node [font=\tiny,  below=-1.2mm, text=red]{$k-1$}(v10);
       
\draw (v0)--node[font=\tiny, below=-1.2mm,text=red]{$2$}(v11);
\draw (v0)--node[font=\tiny, below=-1.2mm,text=red]{$3$}(v12);
\draw (v0)--node[font=\tiny, below=-1.2mm,text=red]{$4$}(v13);
\draw (v0)--node[font=\tiny, below=-1.2mm,text=red]{$k$}(v1-1);
\draw (v0)--node[font=\tiny, below=-1.2mm,text=red]{$1$}(v10);

        
\end{tikzpicture}
    \caption{$v_0$ and its neighbors}
    \label{fig:v0}
\end{figure}
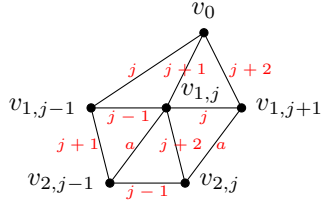
\begin{figure}[htb]
    \centering
    \begin{tikzpicture}[scale=1, wvertex/.style={circle, draw=red, fill=red, scale=0.3}, bvertex/.style={circle, draw=black, fill=black, scale=0.3},rvertex/.style={circle, draw=red, fill=red, scale=0.3}, sbvertex/.style={circle, draw=black, fill=black, scale=0.1}]
    
    \node [bvertex, label={[font=\small] above=2mm:$v_0$}] (v0) at (0.5,1) {};

     \node [bvertex, label={[font=\small] right, above=2mm:$v_{1,j}$}] (v1j) at (0,0) {};

     \node [bvertex, label={[font=\small] right:$v_{1,j+1}$}] (v1j+1) at (1,0) {};
     
     \node [bvertex, label={[font=\small] left:$v_{1,j-1}$}] (v1j-1) at (-1,0) {};
     
     \node [bvertex, label={[font=\small] left:$v_{2,j-1}$}] (v2j-1) at (-0.75,-1) {};
     
     \node [bvertex, label={[font=\small]  right:$v_{2,j}$}] (v2j) at (0.25,-1) {};
     
\draw (v0)--node[font=\tiny, midway, text=red]{$j+1$}(v1j);
\draw (v0)--node[font=\tiny, left, text=red]{$j$}(v1j-1);
\draw (v0)--node[font=\tiny, right=-0.5mm,text=red]{$j+2$}(v1j+1);

\draw (v1j-1)--node[font=\tiny, below=-1mm,text=red]{$j-1$}(v1j);
\draw (v1j)--node[font=\tiny, below=-1mm,text=red]{$j$}(v1j+1);

\draw (v1j-1)--node[font=\tiny, left=-1.2mm,text=red]{$j+1$}(v2j-1);
\draw (v1j)--node[font=\tiny, left=-1mm,text=red]{$a$}(v2j-1);
\draw (v1j+1)--node[font=\tiny, right=-1mm,text=red]{$a$}(v2j);
\draw (v1j)--node[font=\tiny, right=-3.5mm,text=red]{$j+2$}(v2j);
\draw (v2j-1)--node[font=\tiny, below=-1mm,text=red]{$j-1$}(v2j);
    \end{tikzpicture}
    \caption{$v_{1,j}$ and its neighbors}
    \label{fig:v1j}
\end{figure}
\begin{figure}[htb]
    \centering
    \begin{tikzpicture}[scale=1, wvertex/.style={circle, draw=red, fill=red, scale=0.3}, bvertex/.style={circle, draw=black, fill=black, scale=0.3},rvertex/.style={circle, draw=red, fill=red, scale=0.3}, sbvertex/.style={circle, draw=black, fill=black, scale=0.1}]
    
    \node [bvertex, label={[font=\small] left, above=2mm:$v_{i-1,j}$}] (vi-1j) at (-0.25,1) {};
    \node [bvertex, label={[font=\small] right, above=1mm:$v_{i-1,j+1}$}] (vi-1j+1) at (0.75,1) {};

     \node [bvertex, label={[font=\small] right, above=2mm:$v_{i,j}$}] (vij) at (0,0) {};

     \node [bvertex, label={[font=\small] right:$v_{i,j+1}$}] (vij+1) at (1,0) {};
     
     \node [bvertex, label={[font=\small] left:$v_{i,j-1}$}] (vij-1) at (-1,0) {};
     
     \node [bvertex, label={[font=\small] left:$v_{i+1,j-1}$}] (vi+1j-1) at (-0.75,-1) {};
     
     \node [bvertex, label={[font=\small]  right:$v_{i+1,j}$}] (vi+1j) at (0.25,-1) {};

\draw (vij)--node[font=\tiny,right=-1mm,text=red]{$a$}(vi-1j+1);
\draw (vi-1j)--node[font=\tiny, above=-1mm, text=red]{$j$}(vi-1j+1);     
\draw (vi-1j)--node[font=\tiny, midway, text=red]{$i+j$}(vij);
\draw (vi-1j)--node[font=\tiny, left, text=red]{$a$}(vij-1);
\draw (vi-1j+1)--node[font=\tiny, right=-1.5mm,text=red]{$i+j+1$}(vij+1);

\draw (vij-1)--node[font=\tiny, below=-1mm,text=red]{$j-1$}(vij);
\draw (vij)--node[font=\tiny, below=-1mm,text=red]{$j$}(vij+1);

\draw (vij-1)--node[font=\tiny, left=-1.2mm,text=red]{$i+j$}(vi+1j-1);
\draw (vij)--node[font=\tiny, left=-0.5mm,text=red]{$b$}(vi+1j-1);
\draw (vij+1)--node[font=\tiny, right=-0.5mm,text=red]{$b$}(vi+1j);
\draw (vij)--node[font=\tiny, midway,text=red]{$i+j+1$}(vi+1j);
\draw (vi+1j-1)--node[font=\tiny, below=-1mm,text=red]{$j-1$}(vi+1j);
    \end{tikzpicture}
    \caption{$v_{i,j}$ (for even $i\in [2,k-4]$) and its neighbors}
    \label{fig:vij}
\end{figure}
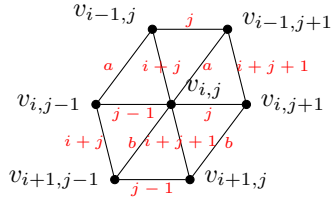
Suppose $v=v_0$.  Recall that $v_0$ has degree $k$. By the definition of $\sigma$, we have $\sigma(v_0v_{1,j})=\sigma(v_{0,j}v_{1,j})=j+1$ for every $j\in [k]$, and hence the edges incident with $v_0$ have distinct colors. We may now assume that $v=v_{1,j}$, which is a degree-$5$ vertex. Note that $\sigma(v_{1,j}v_{1,j+1})=j, \sigma(v_{1,j}v_{1,j-1})=j-1, \sigma_k(v_{1,j}v_0)=j+1, \sigma(v_{1,j}v_{2,j})=j+2$, and $\sigma(v_{1,j}v_{2,j-1})=a$. Observe that $j, j-1, j+1, j+2, a$ are pairwise distinct. Similarly, we can show that all edges incident with $v_{k-2}$ or $v_{k-3, j}$ have distinct colors. Now we consider $v_{i,j}$ for $2\le i \le k-4$ if $k\ge 6$. Note that $v_{i,j}$ is adjacent to $v_{i,j-1}, v_{i,j+1}, v_{i-1,j}, v_{i+1,j}, v_{i-1,j+1}, v_{i+1,j-1}$ and that all the six edges receive colors $j-1,j,i+j, i+j+1, a, b$. Since $2\le i \le k-4$,  the colors $j-1,j,i+j, i+j+1$ are pairwise distinct. This implies that $\sigma$ is a proper edge-coloring.

Now we show that $H_k$ contains no rainbow $C_4$ under $\sigma$. Since $H_k$ is $5$-connected, every $4$-cycle of $H_k$ is formed by two adjacent facial triangles and so it is corresponding to a unique edge of $H_k$. For each $v\in V(H_k)$ and each edge $e$ incident with $v$, it is not hard to check that the $4$-cycle determined by $e$ is not rainbow. Hence, $H_k$ has a proper edge-coloring with no rainbow $C_4$.
\end{proof}

\section{Rainbow planar Tur{\'a}n number of $C_5,C_6$}
In this section, we define an $n$-vertex planar triangulation $F_n$ for each $n\ge 4$. For each even $n$ with $n\ge 4$,  let $p:=\lfloor\frac{n}{2}\rfloor=\frac{n}{2}$ and $F_n$ be defined as follows.
$$V(F_n)=\{ u_1, u_2, \ldots, u_{p}, v_1, v_2,\ldots, v_{p}\},$$
and $$F_n=P_u\cup P_v\cup Q_1\cup Q_2,$$ where 
$P_u, P_v, Q_1, Q_2$ are paths and $$P_u=u_1u_2\ldots u_{p},$$ $$P_v=v_1v_2\ldots v_{p},$$ 
$$Q_1=u_1v_1u_2v_2\ldots u_{p-1}v_{p-1}u_{p}v_{p},$$ $$Q_2=v_2u_1v_3u_2\ldots v_{p-1}u_{p-2}v_{p}u_{p-1}.$$
It follows that for each $i\in [p]$, $u_i$ is adjacent to $u_j$ for $j\in [p]$ with $|j-i|=1$ and $v_j$ for $j\in [p]$ with $-1\le j-i\le 2$; and $v_i$ is adjacent to $v_j$ for $j\in [p]$ with $|j-i|=1$ and $u_j$ for $j\in [p]$ with $-2\le j-i\le 1$.

\begin{figure}[htb]
    \centering
    \resizebox{6cm}{!}
    {\begin{tikzpicture}[scale=1, wvertex/.style={circle, draw=red, fill=red, scale=0.3}, bvertex/.style={circle, draw=black, fill=black, scale=0.3},rvertex/.style={circle, draw=red, fill=red, scale=0.3}, sbvertex/.style={circle, draw=black, fill=black, scale=0.1}]

  \node [bvertex, label={[font=\small] above:$u_1$}] (u1) at (-2.5,0.5) {};
   \node [bvertex, label={[font=\small] above:$u_2$}] (u2) at (-1.5,0.5) {};
    \node [bvertex, label={[font=\small] above:$u_3$}] (u3) at (-0.5,0.5) {};
     \node [bvertex, label={[font=\small] above:$u_{p-1}$}] (u-1) at (1.5,0.5) {};
      \node [bvertex, label={[font=\small] above:$u_p$}] (u0) at (2.5,0.5) {};
    
  \node [bvertex, label={[font=\small] below:$v_1$}] (v1) at (-2.5,-0.5) {};
   \node [bvertex, label={[font=\small] below:$v_2$}] (v2) at (-1.5,-0.5) {};
    \node [bvertex, label={[font=\small] below:$v_3$}] (v3) at (-0.5,-0.5) {};
     \node [bvertex, label={[font=\small] below:$v_{p-1}$}] (v-1) at (1.5,-0.5) {};
      \node [bvertex, label={[font=\small] below:$v_p$}] (v0) at (2.5,-0.5) {}; 

      \draw (u1)--(u2)--(u3);
      \draw (u-1)--(u0);
    \draw (v1)--(v2)--(v3);
      \draw (v-1)--(v0);
      \draw (u1)--(v1)--(u2)--(v2)--(u3)--(v3);
      \draw (u-1)--(v-1)--(u0)--(v0);
      \draw [dashed] (v2)--(u1)--(v3) --(u2);
      \draw [dashed] (u-1)--(v0);
      
      \node (I1) at (-0.5,0) {};
      \draw [dashed] (u2)--(I1);
      \node (I2) at (1.5,0) {};
      \draw [dashed] (v0)--(I2);

      \node(I3) at (0,0.5){};
      \draw (u3)--(I3);

        \node(I4) at (0,-0.5){};
      \draw (v3)--(I4);

        \node(I5) at (1,0.5){};
      \draw (u-1)--(I5);

      \node(I6) at (1,-0.5){};
      \draw (v-1)--(I6);

\node[sbvertex] at (0.5,0) {};
\node[sbvertex] at (0.2,0) {};
\node[sbvertex] at (0.8,0) {};
\end{tikzpicture}}
    
     \caption{$F_n$ for even $n$, where $p=\lfloor\frac{n}{2}\rfloor$}
    \label{fig:C_5_even}
\end{figure}
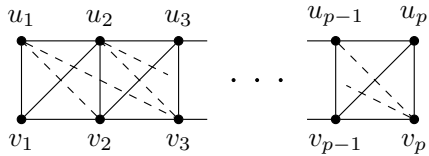
Now we use $F_{n}$, where $n\ge 4$ is even, to define $F_{n+1}$. We add a new vertex $u_0$ and join $u_0$ with $u_1,v_1,v_2$, i.e., insert $u_0$ to the face with boundary $u_1v_1v_2u_1$ and join $u_0$ to those three vertices in the boundary. In fact, $F_{n-1}=F_n-\{v_1\}$ for even $n\ge 4$. 

\begin{figure}[htb]
    \centering
    \resizebox{7cm}{!}
    {\begin{tikzpicture}[scale=1, wvertex/.style={circle, draw=red, fill=red, scale=0.3}, bvertex/.style={circle, draw=black, fill=black, scale=0.3},rvertex/.style={circle, draw=red, fill=red, scale=0.3}, sbvertex/.style={circle, draw=black, fill=black, scale=0.1}]

  \node [bvertex, label={[font=\small] above:$u_1$}] (u1) at (-2.5,0.5) {};
   \node [bvertex, label={[font=\small] above:$u_2$}] (u2) at (-1.5,0.5) {};
    \node [bvertex, label={[font=\small] above:$u_3$}] (u3) at (-0.5,0.5) {};
     \node [bvertex, label={[font=\small] above:$u_{p-1}$}] (u-1) at (1.5,0.5) {};
      \node [bvertex, label={[font=\small] above:$u_p$}] (u0) at (2.5,0.5) {};
    
  \node [bvertex, label={[font=\small] below:$v_1$}] (v1) at (-2.5,-0.5) {};
   \node [bvertex, label={[font=\small] below:$v_2$}] (v2) at (-1.5,-0.5) {};
    \node [bvertex, label={[font=\small] below:$v_3$}] (v3) at (-0.5,-0.5) {};
     \node [bvertex, label={[font=\small] below:$v_{p-1}$}] (v-1) at (1.5,-0.5) {};
      \node [bvertex, label={[font=\small] below:$v_p$}] (v0) at (2.5,-0.5) {}; 

      \draw (u1)--(u2)--(u3);
      \draw (u-1)--(u0);
    \draw (v1)--(v2)--(v3);
      \draw (v-1)--(v0);
      \draw (u1)--(v1)--(u2)--(v2)--(u3)--(v3);
      \draw (u-1)--(v-1)--(u0)--(v0);
      \draw [dashed] (v2)--(u1)--(v3) --(u2);
      \draw [dashed] (u-1)--(v0);
      
      \node (I1) at (-0.5,0) {};
      \draw [dashed] (u2)--(I1);
      \node (I2) at (1.5,0) {};
      \draw [dashed] (v0)--(I2);

      \node(I3) at (0,0.5){};
      \draw (u3)--(I3);

        \node(I4) at (0,-0.5){};
      \draw (v3)--(I4);

        \node(I5) at (1,0.5){};
      \draw (u-1)--(I5);

      \node(I6) at (1,-0.5){};
      \draw (v-1)--(I6);

\node[sbvertex] at (0.5,0) {};
\node[sbvertex] at (0.2,0) {};
\node[sbvertex] at (0.8,0) {};

\node[bvertex, label={[font=\small] above:$u_0$}] (u) at (-3.5,0.5) {};
\draw [dashed](u)--(v1);
\draw (u)-- (u1);
\draw  [dashed] (u) -- (v2);
\end{tikzpicture}}
     \caption{$F_n$ for odd $n$, where $p=\lfloor\frac{n}{2}\rfloor$}
    \label{fig:C_5_odd}
\end{figure}
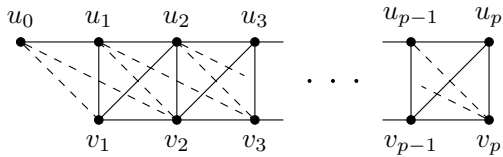
We give a proper edge-coloring of $F_n$ for each $n\ge 3$ with no rainbow $C_5$ or $C_6$ in the following lemma, which derives Theorem~\ref{thm:C5}.
\begin{lemma}
  For each $n\ge 3$, $F_n$ admits a proper edge-coloring with no rainbow $C_5$ or $C_6$.  
\end{lemma}
\begin{proof}
For each $n\ge 3$, let $p=\lfloor \frac{n}{2} \rfloor$. We use the same vertex labels as  Figures~\ref{fig:C_5_even} and~\ref{fig:C_5_odd}. Note that $V(F_n)=\{u_1, u_2, \ldots, u_p, v_1, v_2, \ldots, v_p\}$ for even $n$ and $V(F_n)=\{u_0\}\cup \{u_1, u_2, \ldots, u_p, v_1, v_2, \ldots, v_p\}$ for odd $n$.  We give a proper $6$-edge-coloring for $F_n$, $c: \E(F_n)\to [6]$ as follows.
$$
c(e) = \begin{cases} 
                        1 & \textrm{ if $e=u_iu_{i+1}$ or $e=v_iv_{i+1}$ with odd $i$,}\\
                        2 & \textrm{ if $e=u_iu_{i+1}$ or $e=v_iv_{i+1}$ with even $i$,}\\
                        3 & \textrm{ if $e=u_iv_{i+1}$ or $e=v_iu_{i+1}$ with odd $i$,}\\
                        4 & \textrm{ if $e=u_iv_{i+1}$ or $e=v_iu_{i+1}$ with even $i$,}\\
                        5 & \textrm{ if $e=u_iv_i$  for every possible $i$,}\\
                        6 & \textrm{ if $e=u_iv_{i+2}$ for every possible $i$.}
                    \end{cases}
$$
It is obvious that the coloring $c$ is proper. We claim that $F_n$ with the coloring $c$ has no rainbow $C_5$ or $C_6$. Suppose $F_n$ contains a rainbow $C_5$, say $D$. We show that $D$ must contain all colors in $\{3,4,5,6\}$. Let $X_v=\{v_1, v_2, \ldots, v_p\}$ and $X_u=V(F_n)\backslash X_v$. Observe that all edges in $F_n[X_u]$ or $F_n[X_v]$ are colored $1$ or $2$, and so it follows that $V(D)\cap X_u\neq \emptyset$ and $V(D)\cap X_v\neq \emptyset$. Observe that $D$ must have even number of edges with one end in $X_u$ and one end in $X_v$, and hence $D$ has exactly four such edges (as $D$ is a rainbow $C_5$). This implies that $D$ uses all the colors in $\{3,4,5,6\}$ and only one color in $\{1,2\}$. By the construction for $F_n$, we have that the distance between $u_i$ and $u_{i+3}$ is exactly two and $u_iv_{i+2}u_{i+3}$ is the unique $u_iu_{i+3}$-path  of length two. Similarly, we obtain that $d(v_i,v_{i+3})=2$ with the unique $v_iv_{i+3}$-path $v_iu_{i+1}v_{i+3}$ having length two, $d(u_i, v_{i+3})=2$ with exactly two $u_iv_{i+3}$-paths $u_iv_{i+2}v_{i+3}, u_iu_{i+1}v_{i+3}$ of length two, and $v_i, u_{i+3}$ have distance greater than two as there is no $v_iu_{i+3}$-path of length two in $F_n$. Moreover, for $i,j$ with $|i-j|>3$, we have that $d(u_i,u_j)>2, d(v_i,v_j)>2$, and $d(u_i, v_j)>3$. 

Since $D$ is a $5$-cycle, any two vertices in $D$ have distance at most two in $F_n$. Then it follows that for any two vertices in $D$, the absolute value of the difference for their indices cannot exceed three. Next we claim that the difference between the maximum index and the minimum index of the vertices in $D$ is exactly two.  Suppose not and then their difference is three. Assume that $u_i, u_{i+3}\in V(D)$ for some $i$. Since $u_i v_{i+2}u_{i+3}$ is the unique $u_iu_{i+3}$-path of length two in $F_n$, it follows that $u_iv_{i+2}u_{i+3}\subseteq D$. We claim that $u_{i+3}v_{i+3}\notin E(D)$. Suppose $u_{i+3}v_{i+3}\in E(D)$. Then the other edge in $D$ incident with $v_{i+3}$ is either the edge $u_{i+2}v_{i+3}$ or the edge $u_{i+1}v_{i+3}$. Note that $c(u_{i+2}v_{i+3})=c(v_{i+2}u_{i+3})$ and $c(u_{i+1}v_{i+3})=c(u_{i}v_{i+2})$. Hence $u_{i+3}v_{i+3}$ is not contained in $D$, and so $u_{i+3}u_{i+2}\in E(D)$. Since $D$ uses only one color in $\{1,2\}$, this implies that $D=u_iv_{i+2}u_{i+3}u_{i+2}v_{i+1}u_i$ and the edges $u_iv_{i+1}$ and $v_{i+2}u_{i+3}$ in $D$ have the same color, contradicting that $D$ is rainbow. Hence, both $u_i, u_{i+3}$ cannot be contained in $V(D)$. By symmetry, we have that both $v_i, v_{i+3}$ cannot be contained in $V(D)$. We may now assume that $u_i, v_{i+3}\in V(D)$ (as $v_i, u_{i+3}$ have distance greater than three). It follows that $v_i\notin V(D)$ as $v_{i+3}\in V(D)$, and $u_{i+3}\notin V(D)$ as $u_i\in V(D)$. Suppose $u_iv_{i+2}v_{i+3}\subseteq D$. We claim that $u_iv_{i+1}\in E(D)$. Since $v_{i+2}v_{i+3}\in E(D)$ is colored by color $1$, we have that $u_iu_{i+1}\notin E(D)$ and so $u_iv_{i+1}\in E(D)$. Hence, either $v_{i+1}u_{i+1}v_{i+3}\subseteq D$ or $v_{i+1}u_{i+2}v_{i+3}\subseteq D$. Observe that $c(u_{i+1}v_{i+3})=c(u_iv_{i+2})$ and $c(u_{i+2}v_{i+3})=c(u_iv_{i+1})$. Therefore, $u_iv_{i+2}v_{i+3}$ cannot be contained in $D$. Similarly, we have a contradiction if $u_iu_{i+1}v_{i+3}\subseteq D$. This implies that $u_i, v_{i+3}\in V(D)$ is not possible. Thus, the difference of the indices of any two vertices in $D$ has absolute value at most two. It follows that either $|V(D)\cap X_u|=3, |V(D)\cap X_v|=2$ or $|V(D)\cap X_v|=3, |V(D)\cap V(X_u)|=2$. Without loss of generality, we may assume that $u_i,u_{i+1}, u_{i+2}\in V(D)$ for some $i$. Then $D$ contains exactly two vertices in $v_i, v_{i+1}, v_{i+2}$. Observe that $u_i,u_{i+1}, u_{i+2}$ are not consecutive in $D$ as $D$ cannot use two edges with colors $1$ or $2$. Hence, two vertices in $V(D)\cap X_v$ are not consecutive in $D$. It follows that either $u_iu_{i+1}\in E(D)$ or $u_{i+1}u_{i+2}\in E(D)$. Since $D$ has all the colors in $\{3,4,5,6\}$, the only edge $u_iv_{i+2}$ with color $6$ must be contained in $D$. Suppose $u_iu_{i+1}\in E(D)$. This implies that $v_{i+2}u_{i+2}\in E(D)$ and $D=u_iv_{i+2}u_{i+2}v_{i+1}u_{i+1}u_i$, which is impossible as $u_{i+1}v_{i+1},u_{i+2}v_{i+2}\in E(D)$ have the same color. We may now assume that $u_{i+1}u_{i+2}\in E(D)$. 
If $u_{i+1}v_{i+2}\in E(D)$ then $D=u_iv_{i+2}u_{i+1}u_{i+2}v_{i+1}u_i$ and $u_{i+1}v_{i+2},v_{i+1}u_{i+2}$ have the same color, giving a contradiction. Thus $u_{i+2}v_{i+2}\in E(D)$ and $u_iv_{i+2}u_{i+2}u_{i+1}\subseteq D$. Then either $u_iv_iu_{i+1}\subseteq D$ or $u_iv_{i+1}u_{i+1}\subseteq D$, but both cases imply that there exist two edges in $D$ with the same color $5$. Therefore, $F_n$ has no rainbow $C_5$ under the coloring $c$.

\medskip
Now we show that $F_n$ under the coloring $c$ contains no rainbow $C_6$. Suppose $D$ is a rainbow $6$-cycle in $F_n$. We claim that the difference of the indices of any two vertices in $D$ has absolute value at most three. Suppose $x_i,x_j\in V(D)$, where $j-i \ge 4$ and $x_i\in \{u_i,v_i\}, x_j\in \{u_j, v_j\}$. Note that the distance between $x_i,x_j$ is at least three in $F_n$. Since every two vertices in a $6$-cycle have distance at most three, we know that $d(x_i,x_j)=3$ and $D$ has two internally vertex-disjoint $x_ix_j$-paths of length three. Since $j-i\ge 4$, every $x_ix_j$-path of length three must contain an edge $u_sv_{s+2}$ for some $s$. Since all the edges $u_tv_{t+2}$ are colored by the color $6$, it follows that $D$ has two edges with the same color $6$, giving a contradiction. Thus the difference between the maximum index and the minimum index of the vertices in $D$ is at most three. Suppose the difference is exactly two. Then $V(D)=\{u_i,u_{i+1}, u_{i+2}, v_i, v_{i+1}, v_{i+2}\}$ for some $i$. Observe that $u_iv_{i+2}\in E(D)$ as this is the only one edge in $F_n[D]$ with color $6$. We claim that $v_{i+1}v_{i+2}\notin E(D)$. Assume that $v_{i+1}v_{i+2}\in E(D)$. This implies that the two edges incident with $u_{i+2}$ in $D$ are the edges $u_{i+1}u_{i+2}$ and $v_{i+1}u_{i+2}$. Since $v_{i+1}v_{i+2}$ and $u_{i+1}u_{i+2}$ have the same color, this gives a contradiction and so $v_{i+1}v_{i+2}\notin E(D)$. Suppose $u_{i+1}v_{i+2}\in E(D)$. Similarly, we have $u_{i+1}u_{i+2},v_{i+1}u_{i+2}\in E(D)$, and observe that $u_{i+1}v_{i+2}$ have the same color as $v_{i+1}u_{i+2}$. Thus we have that $u_{i+2}v_{i+2}\in E(D)$. By symmetry, we know that $u_iv_i\in E(D)$, which has the same color $5$ as the edge $u_{i+2}v_{i+2}\in E(D)$. Therefore, we may assume that the minimum index and the maximum index of the vertices in $D$ are $i$ and $i+3$, respectively, for some $i$. Note that the edges $u_iv_{i+2}, u_{i+1}v_{i+3}$ are the only two edges with color $6$ in the subgraph of $F_n$ induced by $\{u_i, u_{i+1}, u_{i+2}, u_{i+3}, v_i, v_{i+1}, v_{i+2}, v_{i+3}\}$. By symmetry, we may assume that $u_iv_{i+2}\in E(D)$. We first show that $v_i\notin V(D)$. Suppose $v_i\in V(D)$. By a similar argument, we know that $u_iv_i\in E(D)$. Assume that $v_iu_{i+1}\in E(D)$. It follows that $u_{i+1}u_{i+2}\in E(D)$. Since $D$ must contain one vertex in $\{u_{i+3}, v_{i+3\}}$, we have $u_{i+2}v_{i+3}v_{i+2}\subseteq D$ or $u_{i+2}u_{i+3}v_{i+2}\subseteq D$. As $u_{i+2}v_{i+3}, v_{i+2}u_{i+3}$ have the same color as the edge $v_iu_{i+1}\in E(D)$, the edge $v_iu_{i+1}$ is not contained in $D$ and so $v_iv_{i+1}\in E(D)$. This implies that $v_{i+1}u_{i+2}\in E(D)$. Similarly, we have $u_{i+2}v_{i+3}v_{i+2}\subseteq D$ or $u_{i+2}u_{i+3}v_{i+2}\subseteq D$. Since the edges $u_{i+2}u_{i+3}, v_{i+2}v_{i+3}$ have the same color as the edge $v_iv_{i+1}\in E(D)$, it follows that $v_iv_{i+1}\notin E(D)$ and so $v_i\notin V(D)$. We next show that $u_{i+3}\notin V(D)$. Assume that $u_{i+3}\in V(D)$. If $v_{i+2}u_{i+3}\notin E(D)$, then $u_{i+2}u_{i+3}v_{i+3}\subseteq D$. It implies that $v_{i+3}u_{i+1}\in E(D)$ or $v_{i+3}v_{i+2}\in E(D)$. Both give a contradiction as $u_iv_{i+2}\in E(D)$ and $u_{i+2}u_{i+3}\in E(D)$. Thus we may assume $v_{i+2}u_{i+3}\in E(D)$. Then it follows that $u_iv_{i+1}\notin E(D)$ and $u_iu_{i+1}\in E(D)$ as $v_i\notin V(D)$. Hence $u_{i+2}u_{i+3}\notin E(D)$ since it has the same color as $u_iu_{i+1}$, and then $u_{i+3}v_{i+3}\in E(D)$. Note all the other edges incident with $v_{i+3}$, which are $u_{i+2}v_{i+3}, u_{i+1}v_{i+3}, v_{i+2}v_{i+3}$, cannot be contained in $D$. Therefore, $u_{i+3}\notin V(D)$ and $V(D)=\{u_i, u_{i+1}, u_{i+2}, v_{i+1}, v_{i+2}, v_{i+3}\}$. Since $u_{i+1}v_{i+3}\notin E(D)$, we have that $v_{i+2}v_{i+3}u_{i+2}\subseteq D$. Note that either $u_iv_{i+1}$ or $u_iu_{i+1}$ is contained in $E(D)$. But $u_iv_{i+1}$ has the same color as $u_{i+2}v_{i+3}\in E(D)$, and $u_iu_{i+1}$ has the same color as $v_{i+2}v_{i+3}\in E(D)$. Hence, $F_n$ cannot contain a rainbow $C_6$ under the coloring $c$. This completes the proof.
\end{proof}

\section*{Acknowledgments}
The author wants to thank Zi-Xia Song for proposing this problem, early discussions and helpful suggestions.

\end{document}